\documentclass[12 pt]{article}
\usepackage{wrapfig}
\usepackage{graphicx}
\usepackage{float}
\usepackage[T1]{fontenc}
\usepackage{amsthm,amssymb,amsmath}
\usepackage{fullpage}
\usepackage{thmtools, thm-restate}
\declaretheorem[numbered=no,name=Chernoff Bound]{chernoff}
\usepackage{color}
\declaretheorem[name=Theorem]{theorem}
\declaretheorem[name=Lemma,numberlike=theorem]{lemma}

\declaretheorem[name=Observation,numberlike=theorem]{observation}
\declaretheorem[name=Corollary,numberlike=theorem]{corollary}

\title{Increasing paths in edge-ordered graphs: \\the hypercube and random graph\thanks{Research supported in part by NSF grant DMS-1427526, "The Rocky Mountain-Great Plains Graduate Research
Workshop in Combinatorics".}}
\author{
Jessica De Silva
\thanks{Department of Mathematics, University of Nebraska, Lincoln, NE, USA, \texttt{jessica.desilva@huskers.unl.edu}. Research is supported by the National Science Foundation Graduate Research Fellowship Program under Grant No. DGE-1041000.}
\and
Theodore Molla
\thanks{Department of Mathematics, University of Illinois at Urbana-Champaign, Urbana, IL, USA, \texttt{molla@illinois.edu}.}
\and
Florian Pfender
\thanks{Department Of Mathematical \& Statistical Sciences, University of Colorado Denver, Denver, CO, USA,
\texttt{florian.pfender@ucdenver.edu.} Research is partially supported by a collaboration grant from the Simons Foundation.}
\and
Troy Retter
\thanks{Department of Mathematics \& Computer Science, Emory University, Atlanta, GA, USA,
\texttt{tretter@emory.edu.}}
\and
Michael Tait
\thanks{Department of Mathematics, University of California, San Diego, CA, USA, \texttt{mtait@ucsd.edu}.}
}
\date{}
\begin{document}
\maketitle


\abstract{An \emph{edge-ordering} of a graph $G=(V,E)$ is a bijection $\phi:E\to\{1,2,\ldots,|E|\}$. Given an edge-ordering, a sequence of edges $P=e_1,e_2,\ldots,e_k$ is an \emph{increasing path} if it is a path in $G$ which satisfies $\phi(e_i)<\phi(e_j)$ for all $i<j$. For a graph $G$, let $f(G)$ be the largest integer $\ell$ such that every edge-ordering of $G$ contains an increasing path of length $\ell$. The parameter $f(G)$ was first studied for $G=K_n$ and has subsequently been studied for other families of graphs. This paper gives bounds on $f$ for the hypercube and the random graph $G(n,p)$.}


\section{Introduction}

An \emph{edge-ordering} of a graph $G=(V,E)$ is a bijection $\phi:E\to\{1,2,\ldots,|E|\}$. Given a graph $G$ and an edge-ordering $\phi$, a sequence of edges $P=e_1,e_2,\ldots,e_k$ is an \emph{increasing path} (of length $k$) if it is a path in $G$ which satisfies $\phi(e_i)<\phi(e_j)$ for all $i<j$. Let $\psi(G,\phi)$ denote the length of the longest increasing path in $G$ with edge-ordering $\phi$. We define
\[
f(G):= \min_{\phi} \psi(G,\phi),
\]
where the minimum is taken over all edge-orderings $\phi$ of $G$. Hence $f(G) \geq \ell$ if every edge-ordering of $G$ contains an increasing path of length $\ell$ and $f(G) < \ell$ if there exists an edge-ordering of $G$ that does not have an increasing path of length $\ell$.

The parameter $f$ was first introduced in 1971 by Chv\'{a}tal and Koml\'{o}s \cite{CK}, who raised the question of estimating $f(K_n)$. Two years later, Graham and Kleitman \cite{GK} established that
\begin{equation}\label{graham and kleitman bound}
\frac{1}{2}\left(\sqrt{4n-3} - 1 \right) \leq f(K_n) \leq \frac{3}{4}n,
\end{equation}
and conjectured $f(K_n)$ should be closer to the upper bound. The upper bound in \eqref{graham and kleitman bound} was improved by Alspach, Heinrich, and Graham to $\frac{7}{12}n$ (unpublished, see \cite{CCS}). Finally, Calderbank, Chung, and Sturtevant \cite{CCS} proved in 1984 that
\[
f(K_n) < \left(\frac{1}{2}+o(1)\right)n.
\]
The sizable remaining gap between this upper bound and the lower bound in \eqref{graham and kleitman bound} has not been improved in the last 40 years.

Although progress has not been made for the complete graph in some time, $f$ has recently been investigated for other classes of graphs. In 2010, Roditty, Shoham, and Yuster \cite{RSY} gave bounds on $f$ for some classes of graphs including trees and planar graphs. In the same year, Katreni\v{c} and Semani\v{s}in \cite{KS} showed that computing $f$ is $NP$-hard in general and that deciding if there is an increasing Hamiltonian path given an edge-ordering is $NP$-complete. In 2001, Yuster \cite{Y} and Alon \cite{A} considered the problem of maximizing $f(G)$ where $G$ ranges over all graphs of maximum degree $d$. Current research by Lavrov and Loh \cite{LL} considers a probabilistic variant that asks for the length of the longest increasing path likely to be present in a random edge-ordering of the complete graph.

This paper contributes to the work on the parameter $f$ by studying it for the hypercube and the random graph $G(n,p)$. We will prove a pair of general lemmas and a pair of resulting theorems.

Before stating these results, however, we make the following (likely well-known) observation.

\begin{observation} \label{lowerBound}
For any graph $G$, let $\chi'(G)$ denote the edge chromatic number of $G$, i.e., the number of matchings needed to cover the edge set of $G$. Then $$f(G) \leq \chi'(G).$$
\end{observation}

\begin{proof}
Let $G=(V,E)$ be any graph and $E = E_1 \cup E_2\cup \cdots\cup E_{\chi'(G)}$ be any proper edge-coloring of $G$.  Now consider any edge-ordering $\phi$ that has the property $\phi(e)<\phi(e') $ if $e \in E_i$ and $e' \in E_j$ for some $i < j$. That is, $\phi$ assigns the edges in $E_1$ the lowest values, the edges in $E_2$ the second lowest values, and so on. Because any increasing path in $\phi$ can use at most one edge from each $E_i$, $ \psi (G,\phi) \leq \chi'(G)$. Hence $f(G) \leq \chi'(G)$. In particular, Vizing's Theorem gives the bound $f(G) \leq \chi'(G) \leq \Delta(G)+1$.
\end{proof}

The following lemmas rely on the \emph{pedestrian algorithm}. The algorithm was initially presented as an idea of Friedgut in \cite{W} to count increasing walks and was altered in \cite{LL} for increasing paths. We defer both the statement of the pedestrian algorithm and the proofs of the lemmas to Section \ref{sec:lemmas}.

\begin{lemma} \label{sets}
Let $G$ be a graph and $k\in\mathbb{Z}^+$. If $f(G) < k$, there exist sets $V_1, V_2, \dots, V_n \subseteq V(G)$ such that $|V_i| \leq k$ and $\displaystyle E(G) \subseteq \cup_{i=1}^n E(G[V_i])$.
\end{lemma}

\begin{lemma} \label{upperBound}
Let $G$ be any connected graph with average degree $d$. For a positive integer $k$, define $\displaystyle \zeta_k := \max_{U \in {V(G) \choose k}} |E(G[U])|$. If $G$ and $k$ satisfy $2\zeta_k-k+1 < d $, then $f(G) \geq k$.
\end{lemma}

In particular, if $k \leq \sqrt d$ we have $2\zeta_k-k+1 \leq 2\zeta_k \leq 2 {k \choose 2} <  d$. This gives another proof of the following result, which was first proved inductively by 
R\"odl in~\cite{R}.

\begin{corollary}\cite{R} \label{cor:rodl} If $G$ is any graph with average degree $d$,
\[
f(G) \geq \sqrt{d}.
\]
\end{corollary}

We will now state our two main theorems.

\begin{restatable}
{theorem}{hypercube}\label{hypercube}Let $Q_d$ denote the $d$-dimensional hypercube. For all $d \ge 2$, \[\frac{d}{\log d} \leq f(Q_d) \leq d.\]
\end{restatable}
\noindent It has been conjectured (cf. \cite{regs}) that $f(Q_d) = d$, which remains open.

Our second theorem relates to the random graph $G(n,p)$, obtained from the complete graph $K_n$ by selecting each edge independently with probability $p$. 

\begin{theorem}\label{gnp}
For any function $\omega(n) \to \infty$ and any $p \leq \frac{\log n}{\sqrt{n}}\omega(n)$, with high probability
\[f(G(n,p)) \geq \frac{(1-o(1))np}{\omega(n)\log n }.\]
\end{theorem}

\begin{corollary}\label{gnp root n}
For any function $\omega(n)\to\infty$ and any $p \geq \frac{\log n}{\sqrt{n}}\omega(n)$, with high probability
\[f(G(n,p)) \geq (1-o(1))\sqrt{n}.\]
\end{corollary}

Noting Graham and Kleitman's lower bound in \eqref{graham and kleitman bound}, we see that graphs far sparser than $K_n$ obtain the same best known bound asymptotically.

In Section \ref{sec:lemmas}, we state the pedestrian algorithm and use it to prove Lemmas \ref{sets} and \ref{upperBound}. In Sections \ref{sec:hc} and \ref{sec:rg}, Theorems \ref{hypercube} and \ref{gnp} are proved respectively.

\section{Proofs of Lemmas \ref{sets} and \ref{upperBound}} \label{sec:lemmas}

We begin this section by stating the \emph{pedestrian algorithm} as presented in \cite{LL}.

\noindent \textbf{Pedestrian Algorithm:}

\emph{Input}: A graph $G$ and an edge-ordering $\phi$.

\emph{Algorithm: }
\begin{enumerate}
\item Place a distinct marker (pedestrian) on each vertex of $G$.
\item Consider the edges in the order given by $\phi$. When an edge $e$ is considered, the pedestrians currently at the vertices incident to $e$ switch places if and only if the switch does not cause either pedestrian to move to a vertex it has already traversed.
\end{enumerate}

Note that at every step in the algorithm there is exactly one pedestrian on each vertex. Also note that each pedestrian traverses an increasing path.

To make use of the pedestrian argument, we find it convenient to introduce the following notation. For a path $P_i$, denote the edge set of the path by $E_i $, the vertex set by $V_i$, and the edges induced by $V_i$ by $U_i$.

\vspace{.3cm} 

\noindent \textbf{Proof of Lemma \ref{sets}:}
 
Take $\phi$ to be an edge-ordering of $G$ that establishes $f(G)<k$. Let $\{p_i\}_{i=1}^n$ be the set of pedestrians. Let $\{P_i\}_{i=1}^n$  to be the increasing paths traversed by the respective pedestrians  $\{p_i\}_{i=1}^n$. This yields the corresponding sets  $\{E_i\}_{i=1}^n$,  $\{V_i\}_{i=1}^n$, and  $\{U_i\}_{i=1}^n$. Since $\phi$ does not have an increasing path of length $k$, we have $|E_i|<k$ for all $i \in [n]$. Clearly $|V_i| = |E_i| + 1 \leq k$.

To prove that $E(G) \subseteq \cup_{i=1}^n U_i$, consider any edge $e \in E(G)$. Let $p_i$ and $p_j$ be the pedestrians located at the vertices incident to $e$ when the edge was considered. Either $p_i$ and $p_j$ switched places so that $e \in U_i$ and $e \in U_j$ or, without loss of generality, $p_i$ had already visited both vertices incident to edge $e$ implying that $e \in U_i$. This completes the proof of Lemma \ref{sets}. \hfill \qed

\vspace{.3cm}
\noindent  \textbf{Proof of Lemma \ref{upperBound}:}
Recall that for a graph $G$ we defined $d$ to be the average degree and $ \zeta_k = \max_{U \in {V (G)\choose k}} |E(G[U])|$. To establish that $2\zeta_k-k+1 < d $ implies $f(G) \geq k$ we will prove the contrapositive by arguing that $f(G) < k$ implies $2\zeta_k-k+1 \geq d $. As before, take paths $\{P_i\}_{i=1}^n$ corresponding to the pedestrian argument applied to an edge ordering $\phi$ that establishes $f(G)<k$. We will now show that 
\begin{align}
|E(G)| \leq \sum_{i=1}^n \left( |U_i| - \frac{|E_i|}{2} \right).  \label{claimInUpper}
\end{align}
Indeed, observe that if $e \not \in E_i$ for every $i \in [n]$, then the edge $e$ contributes at least one to the sum since it is in at least one $U_i$ (and at most 2). Otherwise if $e \in E_i$ for some $i \in [n]$, it must be the case that $e \in E_j$ for exactly one other distinct $j \in [n]$; this is because $e \in E_i$ corresponds to two pedestrians switching places when $e$ was activated. Thus if $e \in E_i$ for some $i \in [n]$, $e$ contributes exactly one to the sum in \eqref{claimInUpper} as $e$ is in precisely two sets in $\{U_i\}_{i=1}^n$ and two sets in $\{E_i\}_{i=1}^n$. This establishes \eqref{claimInUpper}.

We now claim that for each $i \in [n]$,
\begin{align}
|U_i| - \frac{|E_i|}{2}  \leq  \zeta_{|V_i|} - \frac{|V_i|-1}{2} \leq \zeta_{k} - \frac{k-1}{2}. \label{claimInUpper2}
\end{align}

The first inequality in \eqref{claimInUpper2} is an immediate consequence of the fact that for the path $P_i$, $|E_i| = |V_i|-1$, the edges in $U_i$ span exactly $|V_i|$ vertices, and $\zeta_{|V_i|}$ is defined to be the maximum number of edges induced by $|V_i|$ vertices. The second inequality follows from two facts. First, $|V_i| \leq k$ since by construction each $P_i$ had length less than $k$. Second, for all $m < |V(G)|$ it is the case that $\zeta_m - \frac{m}{2} \leq \zeta_{m+1} - \frac{m+1}{2}$; i.e., connectivity implies every set $M$ of size $m<n$ establishing $|\zeta_m|=|E(G[M])|$ can augmented by adding one adjacent vertex to form a set with at least one additional edge.

In conjunction, \eqref{claimInUpper} and \eqref{claimInUpper2} yield

\begin{align}
|E(G)| \leq n \left( \zeta_{k} - \frac{k-1}{2} \right) .
\label{claimInUpper3}
\end{align}

Multiplying both sides of \eqref{claimInUpper3} by $\frac{2}{n}$ establishes $d \leq 2\zeta_k-k+1$. This completes the proof of Lemma \ref{upperBound}. \hfill \qed

\section{The Hypercube} \label{sec:hc}
In this section, we will use Lemma 1 to prove the upper bound for $f(Q_d)$. All logarithms presented in this section are base $2$. Recall that 
\[\displaystyle \zeta_k(G) = \max_{U \in {V(G) \choose k}} |E(G[U])|.\]
 The following lemma is a corollary of a result in \cite{H} (see the theorem and following discussion on pages 131-132). We provide a simple proof by induction of this result for completeness.

\begin{lemma}\label{hypercube density}\cite{H}
For $k,d  \in \mathbb{Z}^+$, the $d$-dimensional hypercube satisfies
\[\zeta_k(Q_d)\leq \frac{k\log k}{2}.\]
\end{lemma}

\begin{proof}
We induct on $d$. For $d=1$ we consider two cases: when $k=1$, $m_1(Q_1) = 0 =  \frac{1\log 1}{2}$ and for $k = 2$, $m_k(Q_1)=1 \leq \frac{k\log k}{2}$.

For $d>1$, consider any $S\subset V(Q_d)$ with $|S| \leq k$. Viewing $Q_d$ as two disjoint copies of $Q_{d-1}$ which are connected by a matching, assume $S$ has $j$ vertices in the first copy of $Q_{d-1}$ and $|S| - j$ vertices in the second. Thus $S$ can induce at most $\mathrm{min}\{ j, |S| - j\}$ edges in the matching. Therefore,
\begin{align*}
\zeta_k(Q_d)&\leq\max_{0\leq j\leq \frac{k}{2}} \left\{ \zeta_j(Q_{d-1})+\zeta_{k-j}(Q_{d-1})+j \right\} \\
&\leq \max_{0\leq j\leq\frac{k}{2}} \left\{ \frac{j\log j}{2}+\frac{(k-j)\log(k-j)}{2}+j  \right\} \quad \text{ (by inductive hypothesis)}
\end{align*}
Now the function $g(j):=\frac{j\log j}{2}+\frac{(k-j)\log(k-j)}{2}+j$ satisfies $g''(j)>0$ on the interval $\left[0,\frac{k}{2}\right]$. Hence its maximum occurs at one of the endpoints. Notice $g(0)=g\left(\frac{k}{2}\right)=\frac{k\log k}{2}$, establishing $\zeta_k(Q_d)\leq \frac{k\log k}{2}$.
\end{proof}

\noindent \textbf{Proof of Theorem \ref{hypercube}:} The upper bound follows from Observation \ref{lowerBound}. 

To prove the lower bound if $ 2 \le d \leq 4$, observe that $\frac{d}{\log d} \leq 2$, so the lower bound claims that an increasing path of length two must exist in every edge-ordering. This is readily obtained in any edge-ordering by considering any two incident edges.
  
We now consider the remaining case when $d > 4$. 
Let $k=\left\lceil\frac{d}{\log d}\right\rceil$. By Lemma \ref{hypercube density}

\begin{equation}\label{thm2hcd}
2\zeta_k(Q_d) - k +1 \leq k\log k - k + 1.
\end{equation}

We claim
\begin{equation}\label{using lemma on Q}
k\log k - k +1 < d.
\end{equation}
Hence \eqref{thm2hcd} and \eqref{using lemma on Q} give $2\zeta_k(Q_d) - k +1 < d$, which by Lemma \ref{upperBound} yields $f(Q_d)\geq k$. Thus to prove the theorem it remains only to verify  \eqref{using lemma on Q} for $d \geq 5$.

If $5\leq d\leq 9$, then $k=3$ and if $10\leq d\leq 16$ then $k=4$, and one can check directly that \eqref{using lemma on Q} holds in either case. Finally if $d> 16$,
\begin{align*}
k\log k - k + 1 & = k (\log k - 1) + 1 \\
& < \left( \frac{d}{\log d}+1 \right) \left(\log d -1 \right) + 1 
&& ( \text{since } d > k \text{ for } d \geq 2 )   \\
& = d + \log d - \frac{d}{\log d} \\
& < d && ( \text{since } d > (\log d)^2 \text{ for } d > 16).
\end{align*}
\hfill \qed

\section{Random Graphs} \label{sec:rg}

This section contains results on the parameter $f$ for the random graph $G(n,p)$. As is common, for convenience we omit any floor and ceiling functions which do not affect the asymptotic nature of our argument.

\vspace{.3cm}
\noindent \textbf{Proof of Theorem \ref{gnp}:}  Consider any fixed function $\omega(n) \to \infty$, any function $p \leq \frac{\log n}{\sqrt{n}}\omega(n)$, and any $\epsilon >0$ that does not depend on $n$. To prove the theorem, we will show that the probability that $G(n,p)$ contains an increasing path of length $k=\frac{(1-\epsilon)np}{\omega(n) \log n}$ approaches 1 as $n \to \infty$. 

Towards this end, let $\mathcal{X}$ be the set of all $n$ element subsets of ${[n] \choose k}$; i.e. $\{V_i\}_{i=1}^{n} \in \mathcal{X}$ if $V_i \subset [n]$ and $|V_i| = k$ for all $i \in [n]$. Clearly each $\cup_{i=1}^{n} K_n[V_i]$ contains at most $n {k \choose 2}$ edges. Hence for any fixed $\{V_i\}_{i=1}^{n} \in \mathcal{X}$, the probability that $G(n,p) \subseteq \cup_{i=1}^{n} K_n[V_i]$ is at most 
\begin{align}
(1-p)^{{n \choose 2}-n {k \choose 2}}. \label{gnpLine}
\end{align}
Now let $\mathcal{B}$ be the set of all (bad) graphs that do not contain an increasing path of length $k$. If $B \in \mathcal{B}$, by Lemma \ref{sets} there exist sets $\{V_i\}_{i=1}^{n} \in \mathcal{X}$ such that 
\[
B \subseteq \bigcup_{i=1}^{n} B[V_i] \subseteq \bigcup_{i=1}^{n} K_n[V_i].\]

It follows from this fact and the union bound that
\begin{align*}
\mathbb{P}\Big(G(n,p) \in \mathcal{B} \Big) &\leq \mathbb{P}\Big(\exists \{V_i\}_{i=1}^n \in \mathcal{X}: G(n,p) \subseteq \bigcup_{i=1}^{n} K_n[V_i] \Big) \\
&\leq \sum_{\{V_i\}_{i=1}^n \in \mathcal{X}} \mathbb{P}\Big( G(n,p) \subseteq \bigcup_{i=1}^{n} K_n[V_i] \Big) \\
& \leq {n \choose k}^n(1-p)^{{n \choose 2}-n {k \choose 2}} \quad \quad \text{(by \eqref{gnpLine}) }\\
&\leq n^{kn}\exp\left\{-p\left({n \choose 2}-n{k \choose 2}\right)\right\} \\
&\leq \exp \Big\{ nk\log n - p{n \choose 2} + pn{k \choose 2} \Big\} \\
&= \exp \Big\{ n \Big( k \log n - \frac{p(n-1)}{2}+p{k \choose 2} \Big) \Big\},
\end{align*}
which approaches zero if $k \log n +p{k \choose 2}  < \frac{p(n-1)}{2}$. By substituting first for $k$ and then for $p$, for $n$ large enough
\begin{align*}
k \log n +p{k \choose 2} &\leq \frac{(1-\epsilon)np}{\omega(n)} + \frac{(1-\epsilon)^2n^2p^3}{2 (\omega(n))^2 (\log n)^2} \\
&\leq \frac{(1-\epsilon)np}{\omega(n)} + \frac{(1-\epsilon)^2 n p}{2} < \frac{p(n-1)}{2},
\end{align*}
which establishes Theorem \ref{gnp}. \hfill \qed

\vspace{.3cm}
\noindent  \textbf{Proof of Corollary \ref{gnp root n}:}
Note that because the function $f$ is monotone with respect to subgraphs, if $f(G(n, p_1)) \geq g(n)$ with high probability for some function $g$, then $f(G(n, p_2)) \geq g(n)$ with high probability for any $p_2\geq p_1$. Therefore, it suffices to consider $p=\frac{\log n}{\sqrt{n}} \omega(n)$ for $\omega(n)\to \infty$ arbitrarily slowly. The lower bound $f(G(n,p)) \geq (1-o(1)) \sqrt{n}$ now follows from Theorem \ref{gnp}.
 \hfill \qed
 \medskip

We remark here that Theorem \ref{gnp} is tight up to a logarithmic factor for many values of $p$. To see this, a standard application of the Chernoff and union bounds (cf. \cite{AS}) gives that the maximum degree of $G(n,p)$ is bounded above by $np(1+o(1))$ with high probability for any $p\geq \frac{\log n}{n}\omega(n)$. Observation \ref{lowerBound} and Theorem \ref{gnp} give that for $\frac{\log n}{n}\omega(n) \leq p \leq \frac{\log n}{\sqrt{n}}\omega(n)$,
\[
\frac{(1+o(1))np}{\omega(n)\log n} \leq f(G(n,p)) \leq (1+o(1))np
\]
with high probability.

\end{document}